\documentclass[10pt,draft,reqno]{amsart}

\makeatletter
\def\section{\@startsection{section}{1}%
\z@{.7\linespacing\@plus\linespacing}{.5\linespacing}%
{\bfseries
\centering
}}
\def\@secnumfont{\bfseries}
\makeatother

\newtheorem{theorem}{Theorem}[section]
\newtheorem{lemma}[theorem]{Lemma}
\newtheorem{proposition}[theorem]{Proposition}

\theoremstyle{definition}
\newtheorem{definition}[theorem]{Definition}
\newtheorem{example}[theorem]{Example}
\theoremstyle{remark}
\newtheorem{remark}[theorem]{Remark}
\numberwithin{equation}{section}
\AtBeginDocument{{\noindent\small
This is a preprint of a paper whose final and definite form is with 
\emph{Global and Stochastic Analysis} (GSA), ISSN 2248-9444, 
available {\tt http://www.mukpublications.com/gsa.php}. 
Submitted 27-Nov-2016; Revised 04-March and 03-April, 2017; Accepted 23-May-2017.}
\vspace{9mm}}

\begin{document}

\title[A Truly Conformable Calculus on Time Scales]{A Truly Conformable Calculus on Time Scales}

\author[B. Bayour]{Benaoumeur Bayour}
\address{Benaoumeur Bayour: 
University of Mascara, B. P. 305, Mamounia Mascara,\newline Mascara 29000, Algeria}
\email{b.benaoumeur@gmail.com}

\author[A. Hammoudi]{Ahmed Hammoudi}
\address{Ahmed Hammoudi:
Laboratoire de Math\'{e}matiques, 
Universit\'{e} de Ain T\'{e}mouchentl,
B. P. 89, 46000 Ain T\'{e}mouchent, Algeria}
\email{hymmed@hotmail.com}

\author[D. F. M. Torres]{Delfim F. M. Torres$^*$}
\thanks{$^*$Corresponding author.}
\address{Delfim F. M. Torres:
Center for Research and Development in Mathematics and Applications (CIDMA),
Department of Mathematics, University of Aveiro, 3810-193 Aveiro, Portugal}
\email{delfim@ua.pt}





\begin{abstract}
We introduce the definition of conformable derivative
on time scales and develop its calculus. Fundamental
properties of the conformable derivative and integral
on time scales are proved. Linear conformable differential 
equations with constant coefficients are investigated,
as well as hyperbolic and trigonometric functions.
\end{abstract}

\maketitle


\section{Introduction}

Local, limit-based, definitions of a so-called
conformable derivative on time scales have been
recently formulated in \cite{[Ben.T]} by
\begin{equation}
\label{eq:[Ben.T]}
T_{\alpha}(f)(t)=\frac{f(\sigma(t))-f(t)}{\sigma(t) - t} \, t^{1-\alpha},
\quad \alpha\in (0,1],
\end{equation}
and then in \cite{[F.X]} by
\begin{equation}
\label{eq:[F.X]}
T_{\alpha}(f)(t)
=\frac{f(\sigma(t))-f(t)}{\sigma(t)^{\alpha}-t^{\alpha}},
\quad \alpha\in (0,1].
\end{equation}
Note that if $f$ is $\Delta$-diferentiable at a
right-scattered point  $t\in \mathbb{T}^{\kappa}_{[0,\infty)}$
\cite{[B.P.1]}, then $f$ is $\alpha$-differentiable in both cases:
for the first definition \eqref{eq:[Ben.T]} we have
\begin{equation}
\label{eq3}
T_{\alpha}(f)(t)=t^{1-\alpha}f^{\Delta}(t)
\end{equation}
while for the second definition \eqref{eq:[F.X]} one has
\begin{equation}
\label{eq03}
T_{\alpha}(f^{\Delta})(t)
=\frac{\sigma(t)-t}{\sigma^{\alpha}(t)-t^{\alpha}} f^{\Delta}(t),
\end{equation}
where $f^{\Delta}(t)=\frac{f^{\sigma}(t)-f(t)}{\sigma(t)-t}$.
The conformable calculus in the time scale $\mathbb{T} = \mathbb{R}$ is now
a well-developed subject: see, e.g., \cite{[T.A],[A.H],MyID:351}
and references therein. For results on arbitrary time scales
see \cite{MyID:330,MR3557866,MR3614829}. However, the adjective \emph{conformable}
may not be appropriate, because $T_{0}f\neq f$, that is, letting $\alpha\rightarrow 0$
does not result in the identity operator. This is also the case
for the recent results of \cite{Zhao:Li}. Moreover, according to \eqref{eq3}
and \eqref{eq03}, the variable $t$ must satisfy $t\geq 0$. With this in mind,
in this paper we extend the calculus of \cite{[D.D]}, by considering a truly
conformable derivative of order $\alpha$, $0\leq\alpha\leq 1$,
on an arbitrary time scale $\mathbb{T}$.


\section{Preliminaries}

We briefly recall the necessary concepts from
the time-scale calculus \cite{[B.P.1],[B.P.2]}.
A time scale $\mathbb{T}$ is an arbitrary nonempty closed
subset of the real numbers $\mathbb{R}$. For $t\in\mathbb{T}$,
the forward jump operator $\sigma:\mathbb{T}\rightarrow\mathbb{T}$
is defined by 
$$
\sigma(t)=\inf\{s\in\mathbb{T}:s>t\}
$$
and the backward jump operator $ \rho:\mathbb{T}\rightarrow\mathbb{T}$ by 
$$
\rho(t):=\sup\{s\in\mathbb{T}:s<t\}. 
$$
The graininess function
$\mu:\mathbb{T}\rightarrow[0,+\infty[$ is given by 
$$
\mu(t)= \sigma(t)-t.
$$
If $\sigma(t)>t$, then $t$ is said right-scattered,
while if $\rho(t)<t$, then $t$ is left-scattered. Moreover,
if $t < \sup\mathbb{T}$  and $\sigma(t)=t$, then $t$ is called right-dense;
if $t>\inf\mathbb{T} $ and $\rho(t)=t$, then $t$ is called left-dense.
If $\mathbb{T} $ has a left-scattered maximum $m$, then
$\mathbb{T}^{\kappa}=\mathbb{T}\setminus\{ m\}$; otherwise,
$\mathbb{T}^{\kappa}=\mathbb{T}$. If $f:\mathbb{T}\rightarrow \mathbb{R}$,
then function $f^{\sigma}:\mathbb{T}\rightarrow \mathbb{R}$ is defined
by $f^{\sigma}=f \circ \sigma$.

\begin{definition}
\label{def:rd:cont}
A function $f: \mathbb{T}\rightarrow \mathbb{R}$ is rd-continuous
provided it is continuous at right-dense points in $\mathbb{T}$ and its
left-sided limits exist (finite) at left-dense points in $\mathbb{T}$.
A function $k:[0,1]\times \mathbb{T}\rightarrow [0,\infty)$
is rd-continuous if $k(\alpha,\cdot): \mathbb{T}\rightarrow [0,\infty)$
is rd-continuous for all $\alpha\in[0,1]$ and $k(\cdot,t):
[0,1]\rightarrow [0,\infty)$ is continuous for all $t\in\mathbb{T}$.
\end{definition}

The set of rd-continuous functions  $f:\mathbb{T}\rightarrow\mathbb{R}$ is
denoted by $C_{rd}$. We say that a function $p:\mathbb{T}\rightarrow\mathbb{R}$
is regressive provided $1+\mu(t)p(t)\neq 0$ holds for all $t \in \mathbb{T}^{\kappa}$.
The set of all regressive and rd-continuous functions
$f:\mathbb{T}\rightarrow\mathbb{R}$ is denoted by $\mathcal{R}$.

The next definition serves as the basis to our notion
of conformable differential operator in Section~\ref{sec:mr}.

\begin{definition}[See \cite{[B.P.1],[B.P.2]}]
\label{def:der:ts}
Assume $f:\mathbb{T}\rightarrow \mathbb{R}$ and let $t\in\mathbb{T}^\kappa$.
We define $f^{\Delta}(t)$ to be the number (provided it exists)
with the property that given any $\epsilon>0$, there is a neighborhood $U$ of $t$
(i.e., $U=(t-\delta,t+\delta)\cap\mathbb{T}$ for some $\delta>0$) such that
$$
\left|[f(\sigma(t))-f(s)]-f^{\Delta}(t)[\sigma(t)-s]\right|
\leq\epsilon\left|\sigma(t)- s\right|
$$
for all $s\in U$. We call $f^{\Delta}(t)$ the delta derivative of $f$ at $t$.
\end{definition}


\section{Main results}
\label{sec:mr}

We begin by introducing the notion
of conformable differential operator
of order $\alpha\in[0,1]$ on an
arbitrary time scale $\mathbb{T}$.

\begin{definition}[Conformable delta differential operator of order $\alpha$]
\label{def:conf:do}
Let $\mathbb{T}$ be a time scale and let $\alpha\in[0,1]$.
An operator $\Delta^{\alpha}$ is conformable if and only if
$\Delta^{0}$ is the identity operator and $\Delta^{1}$
is the standard differential operator on $\mathbb{T}$. Precisely, operator
$\Delta^{\alpha}$ is conformable if and only if for a differentiable function $f$
in the sense of Definition~\ref{def:der:ts}, one has
$\Delta^{0} f=f$ and $\Delta^{1}f=f^{\Delta}$.
\end{definition}

Proposition~\ref{df1} gives an extension of \cite{[D.D]}
to time scales $\mathbb{T}$: for $\mathbb{T} = \mathbb{R}$,
\eqref{eq1} subject to \eqref{eq2} gives
Definition~1.3 of \cite{[D.D]}.

\begin{proposition}[A conformable derivative $\Delta^{\alpha}$ on time scales]
\label{df1}
Let $\mathbb{T}$ be a time scale, $\alpha\in[0,1]$, and $\kappa_{0},\kappa_{1}
: [0,1]\times \mathbb{T}\rightarrow [0,\infty)$ be rd-continuous functions
(see Definition~\ref{def:rd:cont}) such that
\begin{equation}
\label{eq2}
\begin{gathered}
\lim_{\alpha\rightarrow 0^{+}}\kappa_{1}(\alpha,t) = 1,
\quad \lim_{\alpha\rightarrow 0^{+}}\kappa_{0}(\alpha,t) =0,\\
\lim_{\alpha\rightarrow 1^{-}}\kappa_{1}(\alpha,t) = 0,
\quad \lim_{\alpha\rightarrow 1^{-}}\kappa_{0}(\alpha,t) =1,\\
\kappa_{1}(\alpha,t)\neq0, \quad \kappa_{0}(\alpha,t)\neq 0,
\quad \alpha\in (0,1],
\end{gathered}
\end{equation}
for all $t\in \mathbb{T}$. Then, the differential operator $\Delta^{\alpha}f$,
defined by
\begin{equation}
\label{eq1}
\Delta^{\alpha}f(t) = \kappa_{1}(\alpha,t)f(t) +\kappa_{0}(\alpha,t)f^{\Delta}(t)
\end{equation}
in the class of $\Delta$-differentiable functions $f$,
is conformable in the sense of Definition~\ref{def:conf:do}.
\end{proposition}

\begin{proof}
The result is a trivial consequence of \eqref{eq2}--\eqref{eq1}:
$\Delta^{0}f = f$ and $\Delta^{1}f = f^\Delta$.
\end{proof}

\begin{remark}
The parameter $\alpha$ has a crucial role.
Indeed, $\alpha$ is the order of the operator.
Note that given a concrete conformable operator, $\alpha$
is a given fixed constant between zero and one. For example,
if $\alpha = 1$, then we get the Hilger derivative \cite{[B.P.1]}.
\end{remark}

\begin{remark}
Let $\alpha\in(0,1]$, $\kappa_{1}(\alpha,t)\equiv 0$,
and $\kappa_{0}(\alpha,t)=t^{1-\alpha}$.
Then, formally, we recover \eqref{eq3} from \eqref{eq1}.
However, such choice of $\kappa_{0}$ and $\kappa_{1}$
is not allowed by \eqref{eq2}
because \eqref{eq3} is not conformable in agreement
with Definition~\ref{def:conf:do}.
\end{remark}

Many examples of conformable derivatives
on time scales are easily obtained
from Proposition~\ref{df1}.

\begin{example}
\label{ex:cd:1}
One can take $\kappa_{1}\equiv (1-\alpha)\omega^{\alpha}$
and $\kappa_{0}\equiv \alpha\omega^{1-\alpha}$
for $\omega\in(0,\infty)$ or $\kappa_{1}(\alpha,t) = (1-\alpha)|t|^{\alpha}$
and $\kappa_{0}(\alpha,t) = \alpha|t|^{1-\alpha}$ on $\mathbb{T}\setminus \{0\}$
in Proposition~\ref{df1}. In this last case,
 \begin{equation*}
 \Delta^{\alpha}f(t)= (1-\alpha)|t|^{\alpha}f(t)+\alpha|t|^{1-\alpha}f^{\Delta}(t).
 \end{equation*}
\end{example}

\begin{example}
\label{ex:cd:2}
 Similarly to Example~\ref{ex:cd:1},
 \begin{equation*}
 \Delta^{\alpha}f(t)= \cos\left(\frac{\alpha \pi}{2}\right)|t|^{\alpha}f(t)
 +\sin\left(\frac{\alpha\pi}{2}\right) |t|^{1-\alpha}f^{\Delta}(t)
 \end{equation*}
is a conformable derivative.
\end{example}

\begin{remark}
\label{Remark:3.6}
Let $\alpha,\beta\in[0,1]$. Note that, in general,
$\Delta^{\alpha}\Delta^{\beta}\neq\Delta^{\beta}\Delta^{\alpha}$.
Indeed, let functions $\kappa_{i}$, $i=0,1$, be $\Delta_{t}$-differentiable
and continuous with respect to $\alpha$ and $f$ be twice $\Delta$-differentiable.
We have $\Delta^{\alpha}f(t)= \kappa_{1}(\alpha,t)f(t)+\kappa_{0}(\alpha,t)f^{\Delta}(t)$
and $\Delta^{\beta}f(t)= \kappa_{1}(\beta,t)f(t)+\kappa_{0}(\beta,t)f^{\Delta}(t)$.
Therefore,
\begin{equation*}
\begin{split}
\Delta^{\beta}\Delta^{\alpha}f(t)
&=\kappa_{1}(\beta,t)\left( \kappa_{1}(\alpha,t)f(t)
+\kappa_{0}(\alpha,t)f^{\Delta}(t) \right)\\
& \quad +\kappa_{0}(\beta,t)\left( \kappa_{1}(\alpha,t)f(t)
+\kappa_{0}(\alpha,t)f^{\Delta}(t) \right)^{\Delta}\\
&=\kappa_{1}(\beta,t)\kappa_{1}(\alpha,t)f(t)
+\kappa_{1}(\beta,t)\kappa_{0}(\alpha,t)f^{\Delta}(t)\\
&\quad +\kappa_{0}(\beta,t)\big[\kappa_{1}^{\Delta}(\alpha,t)
f^{\sigma}(t)+\kappa_{1}(\alpha,t)f^{\Delta}(t)\\
&\quad +\kappa_{0}^{\Delta}(\alpha,t)f^{\Delta^{\sigma}}(t)
+\kappa_{0}(\alpha,t)f^{\Delta^{2}}(t)\big] \\
&=\kappa_{1}(\beta,t)\kappa_{1}(\alpha,t)f(t)
+\kappa_{1}(\beta,t)\kappa_{0}(\alpha,t)f^{\Delta}(t)\\
&\quad +\kappa_{0}(\beta,t)\kappa_{1}^{\Delta}(\alpha,t)f^{\sigma}(t)
+\kappa_{0}(\beta,t)\kappa_{1}(\alpha,t)f^{\Delta}(t)\\
&\quad +\kappa_{0}(\beta,t)\kappa_{0}^{\Delta}(\alpha,t)f^{\Delta^{\sigma}}(t)
+\kappa_{0}(\beta,t)\kappa_{0}(\alpha,t)f^{\Delta^{2}}(t).
\end{split}
\end{equation*}
Similar calculations lead us to
\begin{equation*}
\begin{split}
\Delta^{\alpha}\Delta^{\beta}f(t)
&=\kappa_{1}(\alpha,t)\kappa_{1}(\beta,t)f(t)
+\kappa_{1}(\alpha,t)\kappa_{0}(\beta,t)f^{\Delta}(t)\\
&\quad +\kappa_{0}(\alpha,t)\kappa_{1}^{\Delta}(\beta,t)f^{\sigma}(t)
+\kappa_{0}(\alpha,t)\kappa_{1}(\beta,t)f^{\Delta}(t)\\
&\quad +\kappa_{0}(\alpha,t)\kappa_{0}^{\Delta}(\beta,t)f^{\Delta^{\sigma}}(t)
+\kappa_{0}(\alpha,t)\kappa_{0}(\beta,t)f^{\Delta^{2}}(t).
\end{split}
\end{equation*}
If $\kappa_{0}(\alpha,t)\kappa_{1}^{\Delta}(\beta,t)
\neq \kappa_{0}(\beta,t)\kappa_{1}^{\Delta}(\alpha,t)$
or $\kappa_{0}(\beta,t)\kappa_{0}^{\Delta}(\alpha,t)
\neq\kappa_{0}(\alpha,t)\kappa_{0}^{\Delta}(\beta,t)$, then
$\Delta^{\beta}\Delta^{\alpha} f(t)
\neq \Delta^{\alpha}\Delta^{\beta} f(t)$.
\end{remark}

\begin{example}
Let $\mathbb{T}$ be a time scale, $\kappa_{1}\equiv (1-\alpha)\omega^{\alpha}$
and $\kappa_{0}\equiv \alpha\omega^{1-\alpha}$, where $\omega > 0$,
$\alpha=\frac{1}{2}$ and $\beta=1$. If $t \in \mathbb{T}$, $t \geq 0$,
and $\sigma(t)\neq t$, then
\begin{equation*}
\Delta\Delta^{\frac{1}{2}}f(t)
=\frac{1}{2}\left[ \frac{\sigma^{\frac{1}{2}}(t)
-t^{\frac{1}{2}}}{\mu(t)}\left(f^{\sigma}(t)+f^{\Delta^{\sigma}}(t) \right)
+t^{\frac{1}{2}}\left( f^{\Delta}(t)+f^{\Delta^{2}}(t)\right)\right].
\end{equation*}
On the other hand, we have
\begin{equation*}
\Delta^{\frac{1}{2}}\Delta f(t)
=\frac{1}{2}t^{\frac{1}{2}}\left( f^{\Delta}(t)+f^{\Delta^{2}}(t)\right).
\end{equation*}
In this case, $\Delta\Delta^{\frac{1}{2}}f(t)\neq\Delta^{\frac{1}{2}}\Delta f(t)$.
\end{example}

\begin{definition}[Conformable exponential function on time scales]
Let $\alpha\in(0,1]$, $s,t\in \mathbb{T}$ with $s\leq t$ and let
function $p:\mathbb{T}\rightarrow \mathbb{R}$ be rd-continuous.
Let $\kappa_{0},\kappa_{1}: [0,1]\times\mathbb{T}\rightarrow [0,\infty)$
be rd-continuous and satisfy \eqref{eq2} with
$1+\mu(t)\frac{p(t)-\kappa_{1}(\alpha,t)}{\kappa_{0}(\alpha,t)}\neq0$
for all $t\in\mathbb{T}^{\kappa}$. Then, the conformable exponential
function on the time scale $\mathbb{T}$ with respect to $\Delta^{\alpha}$
in \eqref{eq1} is defined to be
\begin{equation}
\label{eq4}
E_{p}(t,s)=e_{\frac{p(t)-\kappa_{1}(\alpha,t)}{\kappa_{0}(\alpha,t)}}(t,s),
\quad E_{0}(t,s)=e_{\frac{-\kappa_{1}(\alpha,t)}{\kappa_{0}(\alpha,t)}}(t,s),
\end{equation}
where $e_{q(t)}(t,s)$ denotes the exponential function on the time scale $\mathbb{T}$
 -- see definition (2.30) in \cite{[B.P.1]}.
\end{definition}

Note that if $\mathbb{T}= \mathbb{R}$, then 
$$
E_{p}(t,s)=e^{\displaystyle \int_{s}^{t}\frac{p(\tau)
-\kappa_{1}(\alpha,\tau)}{\kappa(\alpha,\tau)}d\tau}
\quad \text{ and } \quad
E_{0}(t,s)=e^{-\displaystyle  
\int_{s}^{t}\frac{\kappa_{1}(\alpha,\tau)}{\kappa(\alpha,\tau)}d\tau}.
$$


\subsection{Fundamental properties of the conformable operators}

Using \eqref{eq1} and \eqref{eq4}, we begin by proving
several basic but important results.

\begin{theorem}[Basic properties of conformable derivatives]
\label{lem1}
Let the conformable differential operator $\Delta^{\alpha}$
be given as in \eqref{eq1}, where $\alpha\in[0,1]$.  Let
$\kappa_{0},\kappa_{1}: [0,1]\times\mathbb{T}\rightarrow [0,\infty)$
be rd-continuous and satisfy \eqref{eq2} with
$1+\mu(t)\frac{p(t)-\kappa_{1}(\alpha,t)}{\kappa_{0}}\neq 0$
for all $t\in\mathbb{T}^{\kappa}$. Assume functions $f$ and $g$
are differentiable, as needed. Then,
\begin{enumerate}
\item[(i)] $\Delta^{\alpha}(af+bg)
=a\Delta^{\alpha}(f)+b\Delta^{\alpha}(g)$
for all $a,b\in\mathbb{R}$;

\item[(ii)] $\Delta^{\alpha}c=c\kappa_{1}(\alpha,\cdot)$
for all constants $c\in\mathbb{R}$;

\item[(iii)] $\Delta^{\alpha}(fg)
=f\Delta^{\alpha}(g)+g^{\sigma}\Delta^{\alpha}(f)
-fg^{\sigma}\kappa_{1}(\alpha,\cdot)$;

\item[(iv)] $\Delta^{\alpha}\left(\frac{f}{g}\right)
=\frac{g\Delta^{\alpha}(f)-f\Delta^{\alpha}(g)}{gg^{\sigma}}
+ \frac{f}{g}\kappa_{1}(\alpha,\cdot)$;

\item[(v)] $\Delta^{\alpha}E_{p}(t,s)=p(t)E_{p}(t,s)$ for all $\alpha\in(0,1]$;

\item[(vi)] $\Delta^{\alpha}\left(\int_{a}^{t}
\frac{f(s)E_{0}(t,s)}{\kappa_{0}(\alpha,s)}\Delta s\right)=f(t)E_{0}(\sigma(t),t)$
for all $\alpha\in(0,1]$.
\end{enumerate}
\end{theorem}

\begin{proof}
Relations (i) and (ii) are obvious. From \eqref{eq1},
it also follows (iii)--(vi):
\begin{enumerate}
\item[(iii)]
\begin{equation*}
\begin{split}
\Delta^{\alpha}(fg)
&= \kappa_{0}(fg^{\Delta}+f^{\Delta}g^{\sigma})+\kappa_{1}(fg)\\
&= f\kappa_{0}g^{\Delta}+g^{\sigma}\kappa_{0}f^{\Delta}+\kappa_{1}(fg) \\
&= f(\kappa_{0}g^{\Delta}+\kappa_{1}g)+g^{\sigma}(\kappa_{0}f^{\Delta}
+\kappa_{1}f)-g^{\sigma}\kappa_{1}f \\
&= f\Delta^{\alpha}g+g^{\sigma}\Delta^{\alpha}f-g^{\sigma}\kappa_{1}f;
\end{split}
\end{equation*}

\item[(iv)]
\begin{equation*}
\begin{split}
\Delta^{\alpha}\left(\frac{f}{g}\right)
&= \kappa_{0}\left(\frac{f^{\Delta}g-fg^{\Delta}}{gg^{\sigma}}\right)
+\kappa_{1}\left(\frac{f}{g}\right) \\
&= \frac{\kappa_{0}(f^{\Delta}g-fg^{\Delta})}{gg^{\sigma}}
+\kappa_{1}\left(\frac{f}{g}\right) \\
&=\frac{(\kappa_{0}f^{\Delta}+\kappa_{1}f)g-\kappa_{1}fg
-\left(\kappa_{0}g^{\Delta}+\kappa_{1}g\right)f+f\kappa_{1}g}{gg^{\sigma}}
+\kappa_{1}\left(\frac{f}{g}\right) \\
&= \frac{g\Delta^{\alpha}f-f\Delta^{\alpha}g}{gg^{\sigma}}
+\kappa_{1}\left(\frac{f}{g}\right);
\end{split}
\end{equation*}

\item[(v)]$ \Delta^{\alpha}E_{p(t)}(t,s)$
\begin{equation*}
\begin{split}
\quad\qquad &= \kappa_{1}(\alpha,t)
e_{\frac{p(t)-\kappa_{1}(\alpha,t)}{\kappa_{0}(\alpha,t)}}(t,s)
+\kappa_{0}(\alpha,t)\left(e_{\frac{p(t)
-\kappa_{1}(\alpha,t)}{\kappa_{0}(\alpha,t)}}(t,s)\right)^{\Delta} \\
\quad\qquad &= \kappa_{1}(\alpha,t)\left(e_{\frac{p(t)
-\kappa_{1}(\alpha,t)}{\kappa_{0}(\alpha,t)}}(t,s)\right)
+\kappa_{0}(\alpha,t)\left(\frac{p(t)-\kappa_{1}(\alpha,t)}{\kappa_{0}(\alpha,t)}
e_{\frac{p(t)-\kappa_{1}(\alpha,t)}{\kappa_{0}(\alpha,t)}}(t,s)\right) \\
\quad\qquad &= p(t)e_{\frac{p(t)-\kappa_{1}(\alpha,t)}{\kappa_{0}(\alpha,t)}}(t,s) \\
\quad\qquad &= p(t)E_{p}(t,s);
\end{split}
\end{equation*}

\item[(vi)] we apply Theorem~1.117 of \cite{[B.P.1]}:
\begin{equation*}
\begin{split}
\Delta^{\alpha}&\int_{a}^{t}
\frac{f(s)E_{0}(t,s)}{\kappa_{0}(\alpha,s)}\Delta s
= \kappa_{1}(\alpha,t)\int_{a}^{t}
\frac{f(s)E_{0}(t,s)}{\kappa_{0}(\alpha,s)}\Delta s\\
&\qquad +\kappa_{0}(\alpha,t)
\left(\frac{f(t)E_{0}(\sigma(t),t)}{\kappa_{0}(\alpha,t)}+\int_{a}^{t}
\frac{f(s)(-\frac{\kappa_{1}(\alpha,t)}{\kappa_{0}(\alpha,t)})
E_{0}(t,s)}{\kappa_{0}(\alpha,s)}\Delta s\right) \\
&=\kappa_{1}(\alpha,t)\int_{a}^{t}\frac{f(s)E_{0}(t,s)}{\kappa_{0}(\alpha,s)}
\Delta s+f(t)E_{0}(\sigma(t),t)-\kappa_{1}(\alpha,t)\int_{a}^{t}
\frac{f(s)E_{0}(t,s)}{\kappa_{0}(\alpha,s)}\Delta s\\
&= f(t)E_{0}(\sigma(t),t).
\end{split}
\end{equation*}
\end{enumerate}
The proof is complete.
\end{proof}

\begin{definition}[Conformable integrals of order $\alpha$]
Let $\alpha\in(0,1]$ and let
$\kappa_{0},\kappa_{1}: [0,1]\times\mathbb{T}\rightarrow [0,\infty)$
be rd-continuous and satisfy \eqref{eq2} with
$1+\mu(t)\frac{\kappa_{1}(\alpha,t)}{\kappa_{0}(\alpha,t)}\neq0$
for all $t\in\mathbb{T}^{\kappa}$ and $t_{0}\in\mathbb{T}$.
In light of \eqref{eq4} and items (v) and (vi) of Theorem~\ref{lem1},
we define the conformable antiderivative of order $\alpha$ by
$$
\int \Delta^{\alpha} f(t)\Delta^{\alpha}t
= f(t)+cE_{0}(t,t_{0}),
\quad c\in \mathbb{R}.
$$
The conformable $\alpha$-integral of $f$ over $\mathbb{T}_{[a,t]}$ is defined by
\begin{equation}
\label{eq5}
\int_{a}^{t} f(s) E_{0}(\sigma(t),s) \Delta^{\alpha}s
:= \int_{a}^{t}\frac{f(s)E_{0}(t,s)}{\kappa_{0}(\alpha,s)}\Delta s,
\end{equation}
where on the right-hand side we have the standard
$\Delta$-integral of time scales \cite{[B.P.1],[B.P.2]}.
\end{definition}

\begin{remark}
It follows from \eqref{eq5} that $\displaystyle \Delta^{\alpha}s
=\frac{E_{0}(t,s)}{E_{0}(\sigma(t),s)\kappa_{0}(\alpha,s)}\Delta s$.
\end{remark}

\begin{theorem}[Basic properties of the conformable $\alpha$-integral]
\label{lem2}
Let the conformable differential operator on time scales $\Delta^{\alpha}$
be given as in \eqref{eq1}; the integral be given as in \eqref{eq5};
with $\alpha\in (0,1]$. Let functions $\kappa_{0},\kappa_{1}$
be rd-continuous and satisfy \eqref{eq2} with
$1+\mu(t)\frac{\kappa_{1}(\alpha,t)}{\kappa_{0}(\alpha,t)}\neq 0$
for all $t\in\mathbb{T}^{\kappa}$ and let $f$ and $g$
be $\Delta$-differentiable, as needed. Then,
\begin{itemize}
\item[(i)] the derivative of the definite integral of $f$ is given by
$$
\Delta^{\alpha}\left(\int_{a}^{t}f(s)E_{0}(\sigma(t),s)\Delta^{\alpha}s\right)
=f(t)E_{0}(\sigma(t),t);
$$
\item[(ii)] the definite integral of the derivative of $f$ is given by
$$
\int_{a}^{t} \Delta^{\alpha}[f(s)]E_{0}(t,\sigma(s))\Delta^{\alpha}s
=f(s)E_{0}(t,s)\mid_{s=a}^{s=t}
+\int_{a}^{t}f(s)E_{0}(t,\sigma(s))\kappa_{1}(\alpha,s)\Delta^{\alpha}s;
$$
\item[(iii)] an integration by parts formula is given by
\begin{multline*}
\int_{a}^{b}f(t)\Delta^{\alpha}[g(t)]E_{0}(b,\sigma(t))\Delta^{\alpha}t
=f(t)g(t)E_{0}(b,\sigma(t))\mid_{t=a}^{t=b}\\
+\int_{a}^{b}\left[(g(t)-g(\sigma(t)))[f(t)]\kappa_{1}(\alpha,t)
-g(\sigma(t))\Delta^{\alpha}f(t)
\right]E_{0}(b,\sigma(t))\Delta^{\alpha}t;
\end{multline*}
\item[(iv)] a version of the Leibniz rule for differentiation of an integral
is given by
\begin{multline*}
\Delta^{\alpha}\left[\int_{a}^{t}f(t,s)E_{0}(\sigma(t),s)\Delta^{\alpha}s\right]\\
=f(\sigma(t),t)E_{0}(\sigma(t),t)
+ \int_{a}^{t} \Delta_t^{\alpha}(f(t,s))
E_{0}(t,s) E_{0}(\sigma(t),s) \Delta^\alpha s,
\end{multline*}
where the derivative inside the last integral is with respect to $t$.
\end{itemize}
\end{theorem}

\begin{proof}
The proof of (i) is clear. The integration by parts formula (ii)
follows easily:
\begin{equation*}
\begin{split}
\int_{a}^{t}\Delta^{\alpha}&\left[ f(s)\right]E_{0}(t,\sigma(s))\Delta^{\alpha}s\\
&=\int_{a}^{t}\Delta^{\alpha}\left[ f(s)E_{0}(t,s)\right]\Delta^{\alpha}s
+ \int_{a}^{t}f(s)E_{0}(t,\sigma(s))\kappa_{1}(\alpha,s)\Delta^{\alpha}s\\
&=f(s)E_{0}(t,s)\mid_{s=a}^{s=t}+\int_{a}^{t}
f(s)E_{0}(t,\sigma(s))\kappa_{1}(\alpha,s)\Delta^{\alpha}s.
\end{split}
\end{equation*}
Now we prove (iii):
\begin{equation*}
\begin{split}
\int_{a}^{b} & f(s)\Delta^{\alpha}(g)(s)E_{0}(b,s)\Delta^{\alpha}s\\
&= \int_{a}^{b}\Delta^{\alpha}(fg)(s)E_{0}(b,\sigma(s))\Delta^{\alpha}s\\
&\quad -\int_{a}^{b}g^{\sigma}(s)\left[\Delta^{\alpha}(f)(s)
-f(s)\kappa_{1}(\alpha,s) \right]E_{0}(b,\sigma(s))\Delta^{\alpha}s\\
&= (fg)(s)E_{0}(b,\sigma(s))\vert_{a}^{b}+\int_{a}^{b}(fg)(s)
E_{0}(b,\sigma(s))\kappa_{1}(\alpha,s)\Delta^{\alpha}s\\
&\quad -\int_{a}^{b}g^{\sigma}(s)\left[\Delta^{\alpha}(f)(s)
-f(s)\kappa_{1}(\alpha,s) \right] E_{0}(b,\sigma(s))\Delta^{\alpha}s\\
&=(fg)(s)E_{0}(b,\sigma(s))\vert_{a}^{b}\\
&\quad +\int_{a}^{b}\left[ (fg)(s)\kappa_{1}(\alpha,s)
-g^{\sigma}(s)(\Delta^{\alpha}(f)(s)-f(s)\kappa_{1}(\alpha,s)) \right] 
E_{0}(b,\sigma(s))\Delta^{\alpha}s\\
&=(fg)(s)E_{0}(b,\sigma(s))\vert_{a}^{b}\\
&\quad +\int_{a}^{b}\left[(g(s)-g^{\sigma}(s))\kappa_{1}(\alpha,s)f(s)
-g^{\sigma} (s)\Delta^{\alpha}f(s)\right] E_{0}(b,\sigma(s))\Delta^{\alpha}s.
\end{split}
\end{equation*}
For (iv), we have:
\begin{equation*}
\begin{split}
\Delta^{\alpha}&\left[ \int_{a}^{t}f(t,s)E_{0}(\sigma(t),s)\Delta^{\alpha} s \right]\\
&=\Delta^{\alpha}\int_{a}^{t}f(t,s)E_{0}(\sigma(t),s)
\frac{E_{0}(t,s)}{E_{0}(\sigma(t),s)\kappa_{0}(\alpha,s)}\Delta s\\
&=\Delta^{\alpha}\int_{a}^{t}\frac{f(t,s)E_{0}(t,s)}{\kappa_{0}(\alpha,s)}\Delta s\\
&=\kappa_{0}(\alpha,t)\left( \int_{a}^{t}
\frac{f(t,s)E_{0}(t,s)}{\kappa_{0}(\alpha,s)}\Delta s \right)^{\Delta}
+\kappa_{1}(\alpha,t)\int_{a}^{t}\frac{f(t,s)E_{0}(t,s)}{\kappa_{0}(\alpha,t)}\Delta s\\
&=\kappa_{0}(\alpha,t)\left(\frac{f(\sigma(t),t)E_{0}(\sigma(t),t)}{\kappa_{0}(\alpha,t)}
+\int_{a}^{t}\frac{f^{\Delta}(t,s)E_{0}(t,s)}{\kappa_{0}(\alpha,s)} \Delta s \right)\\
&\quad +\kappa_{1}(\alpha,t)\int_{a}^{t}\frac{f(t,s)E_{0}(t,s)}{\kappa_{0}(\alpha,s)}\Delta s\\
&=f(\sigma(t),t)E_{0}(\sigma(t),t) + \int_{a}^{t}\frac{\Delta_t^{\alpha}(f(t,s))
E_{0}(t,s)}{\kappa_{0}(\alpha,s)}\Delta s,
\end{split}
\end{equation*}
where the derivative inside the last integral is with respect to $t$.
\end{proof}


\subsection{Linear 2nd-order conformable differential equations on time scales}

Let $\mathbb{T}$ be an arbitrary time scale, $\alpha\in[0,1]$,
and let $\kappa_{0},\kappa_{1} :[0,1]\times \mathbb{T}\rightarrow [0,\infty)$
be rd-continuous functions such that \eqref{eq2} holds with
$1+\mu(t)\frac{\kappa_{1}(\alpha,t)}{\kappa_{0}(\alpha,t)}\neq 0$
for all $t\in\mathbb{T}^{\kappa}$. In addition, let $\Delta^{\alpha}$
be as in \eqref{eq1}, and let $t_{0}\in\mathbb{T}$.
In this section we are concerned with the following
linear second-order conformable dynamic equation on time scales
with constant coefficients:
\begin{equation}
\label{eq11}
\Delta^{\alpha}\Delta^{\alpha}y(t)+a\Delta^{\alpha}y(t)+by(t)=f(t),
\quad t\in\mathbb{T}^{\kappa^{2}}_{[t_{0},\infty)},
\end{equation}
where we assume $a,b\in\mathbb{R,}$ $f\in C_{rd}$.
Introduce the operator $L_{2\Delta^{\alpha}} : C_{rd}^{2}\rightarrow C_{rd}$ by
\begin{equation}
\label{eq:op:L2}
L_{2\Delta^{\alpha}}(y)(t)
= \Delta^{\alpha}\Delta^{\alpha}y(t)+a\Delta^{\alpha}y(t) + b y(t)
\end{equation}
for all $t\in\mathbb{T}^{\kappa^{2}}_{[t_{0},\infty)}$.

\begin{lemma}
The operator  $L_{2\Delta^{\alpha}}$ defined by \eqref{eq:op:L2}
is a linear operator, i.e.,
$$
L_{2\Delta^{\alpha}}\left(p y_{1}+q y_{2}\right)
=p L_{2\Delta^{\alpha}}(y_{1})
+ q L_{2\Delta^{\alpha}}(y_{2}),
$$
where $p,q\in\mathbb{R}$ and $y_{1},y_{2}\in  C^{2}_{rd}$.
If $y_{1}$ and $y_{2}$ solve the homogeneous equation
$$
L_{2\Delta^{\alpha}}y=0,
$$
then so does $y=p y_{1}+qy_{2}$, $p,q\in \mathbb{R}$.
\end{lemma}

\begin{proof}
Using (i) of Theorem~\ref{lem1}, we find that
\begin{equation*}
\begin{split}
L_{2\Delta^{\alpha}}&(p y_{1}+q y_{2})(t)\\
&= \Delta^{\alpha}\Delta^{\alpha}(p y_{1}(t)+q y_{2}(t))
+a\Delta^{\alpha}(py_{1}(t)+q y_{2}(t))+b(p y_{1}(t)+q y_{2}(t)) \\
&= pL_{2\Delta^{\alpha}}(y_{1})(t)+q L_{2\Delta^{\alpha}}(y_{2})(t)=0
\end{split}
\end{equation*}
for all $t\in\mathbb{T}^{\kappa^{2}}_{[t_{0},\infty)}$
and all $p,q\in\mathbb{R}$.
\end{proof}

\begin{definition}
Let $a,b\in\mathbb{R}$ and $f\in C_{rd}$.
Equation \eqref{eq11} is called regressive if
$$
\kappa^{2}_{0}-\mu\kappa_{0}(a+2\kappa_{1})
+\mu^{2}(b+a\kappa_{1}+\kappa^{2}_{1})\neq0
$$
for all $t\in \mathbb{T}^{\kappa}$.
\end{definition}

\begin{theorem}[Existence and uniqueness of solution]
Let $t_{0}\in \mathbb{T}^{\kappa}$ and functions $\kappa_{i}(\alpha,t)$,
$i=0,1$, be $\Delta_{t}$-differentiable and continuous with respect to $\alpha$.
Assume that the dynamic equation \eqref{eq11} is regressive. If
$L_{2\Delta^{\alpha}}y(t)=0$ admits two solutions $y_{1}$ and $y_{2}$
with $y_{1}(t)\Delta^{\alpha}y_{2}(t) \neq \Delta^{\alpha}y_{1}(t)y_{2}(t)$
for all $t\in\mathbb{T}^{\kappa^{2}}_{[t_{0},\infty)}$, then the initial value problem
\begin{equation}
\label{eq12}
L_{2\Delta^{\alpha}}y(t)=0, \quad y(t_{0}) = y_0,
\quad \Delta^{\alpha}y(t_{0})=y_{0}^{\alpha},
\end{equation}
where $y_{0}$ and $y_{0}^{\alpha}$ are given constants,
has a unique solution defined on  $\mathbb{T}_{[t_{0},\infty)}$.
\end{theorem}

\begin{proof}
If $y_{1},y_{2}$ are two solutions of $L_{2\Delta^{\alpha}}y(t)=0$,
then $y(t)=py_{1}(t)+qy_{2}(t)$, $p,q\in\mathbb{R}$, is a solution of
$L_{2\Delta^{\alpha}}y(t)=0$. Therefore, we want to see if we can pick
$p$ and $q$ so that $y_{0}=y(t_{0})=py_{1}(t_{0})+qy_{2}(t_{0})$,
$y_{0}^{\alpha}=p \Delta^{\alpha}y_{1}(t_{0})+q\Delta^{\alpha}y_{2}(t_{0})$. Let
$$
\mathbf{M} =
\left(
\begin{array}{ccc}
y_{1}(t_{0}) & y_{2}(t_{0})\\
\Delta^{\alpha}y_{1}(t_{0}) & \Delta^{\alpha}y_{2}(t_{0})\\
\end{array}
\right),
\quad
\mathbf{X} =
\left(
\begin{array}{ccc}
p \\
q \\
\end{array}
\right),
\quad
\mathbf{B} =
\left(
\begin{array}{ccc}
y_{0} \\
y_{0}^{\alpha} \\
\end{array}
\right).
$$
System $M\times X=B$ has a unique solution because
we are assuming matrix $M$ to be invertible.
\end{proof}

\begin{definition}
For two $\Delta^{\alpha}$-differentiable functions on $\mathbb{T}_{[t_{0,\infty})}$
$y_{1}$ and $y_{2}$,  we define the Wronskian $W=W(y_{1},y_{2})$ by
\begin{displaymath}
W(t) = \det
\left(
\begin{array}{ccc}
y_{1}(t) & y_{2}(t)\\
\Delta^{\alpha}y_{1}(t) & \Delta^{\alpha}y_{2}(t)\\
\end{array}
\right),
\quad t\in\mathbb{T}_{[t_{0,\infty})}.
\end{displaymath}
We say that two solutions $y_{1}$ and $y_{2}$ of $L_{2\Delta^{\alpha}}y = 0$
form a fundamental set of solutions for $L_{2\Delta^{\alpha}}y = 0$ provided
$W(y_{1},y_{2})(t)\neq 0$ for all $ t \in \mathbb{T}^{\kappa^2}_{[t_{0,\infty})}$.
\end{definition}

\begin{theorem}
If the pair of functions $y_{1}$ and $y_{2}$ form a fundamental system
of solutions for $L_{2\Delta^{\alpha}}y = 0$,
$t\in\mathbb{T}^{\kappa^{2}}_{[t_{0,\infty})}$, then
\begin{equation}
\label{eq:gen:sol}
y(t)=py_{1}(t)+qy_{2}(t),
\quad p,q\in\mathbb{R},
\end{equation}
is a general solution of $L_{2\Delta^{\alpha}}y = 0$,
$t\in\mathbb{T}^{\kappa^{2}}_{[t_{0,\infty})}$.
In particular,
the solution of the initial value problem \eqref{eq12} is given by
$$
y(t)=\frac{\Delta^{\alpha}y_{2}(t_{0})y_{0}
-y_{2}(t_{0})y_{0}^{\alpha}}{W(y_{1},y_{2})(t_{0})}y_{1}(t)
+\frac{y_{1}(t_{0})y_{0}^{\alpha}
-\Delta^{\alpha}y_{1}(t_{0})y_{0}}{W(y_{1},y_{2})(t_{0})}y_{2}(t).
$$
\end{theorem}

\begin{remark}
By general solution we mean that every function of form \eqref{eq:gen:sol}
is a solution and every solution is of this form.
\end{remark}

\begin{proof}
The proof is similar to the one of Theorem~3.7 of \cite{[B.P.1]}.
\end{proof}


\subsection{Hyperbolic and trigonometric functions}

Now we consider the linear second-order homogeneous dynamic
conformable equation with constant coefficients
\begin{equation}
\label{eq14}
\Delta^{\alpha}\Delta^{\alpha}y(t)+a\Delta^{\alpha}y(t)+by(t)=0,
\quad a,b\in\mathbb{R},
\quad t\in\mathbb{T}^{\kappa^{2}}_{[t_{0},\infty)}.
\end{equation}
We assume \eqref{eq14} to be regressive, i.e.,
$\kappa_{0}-\mu(a+2\kappa_{1})+\mu^{2}(b+a\kappa_{1}
+\kappa^{2}_{1})\neq 0$, $t\in \mathbb{T}^{\kappa}$.
Let $\lambda\in \mathbb{C}$ be such that
$1+ \mu(t)\frac{\lambda-\kappa_{1}(\alpha,t)}{\kappa_{0}(\alpha,t)}\neq 0$,
$t\in\mathbb{T}^{\kappa}$, and $y(t)=E_{\lambda}(t,t_{0})$,
$t\in\mathbb{T}^{\kappa}_{[t_{0,\infty})}$, be
a solution of \eqref{eq14}. If $y(t)= E_{\lambda}(t,t_{0})$, then
$$
\Delta^{\alpha}\Delta^{\alpha}y(t)+a\Delta^{\alpha}y(t)+by(t)
=(\lambda^{2}+a\lambda +b)E_{\lambda}(t,t_{0})
$$
and, because $E_{\lambda}(t,t_{0})\neq 0$, $y(t)=E_{\lambda}(t,t_{0})$
is a solution of \eqref{eq14} if and only if $\lambda$ satisfies
the characteristic equation of \eqref{eq14}:
\begin{equation}
\label{eq15}
\lambda^{2}+a\lambda+b=0.
\end{equation}
The solutions $\lambda_{1}$ and $\lambda_{2}$ of \eqref{eq15}
are given by
\begin{equation}
\label{eq16}
\lambda_{1}=\frac{-a-\sqrt{a^{2}-4b}}{2}
\quad \text{ and } \quad
\lambda_{2}=\frac{-a+\sqrt{a^{2}-4b}}{2}
\end{equation}
and, since \eqref{eq14} is regressive,
$1+ \mu(t)\frac{\lambda_{1}-\kappa_{1}(\alpha,t)}{\kappa_{0}(\alpha,t)}\neq0$
and $1+ \mu(t)\frac{\lambda_{2}-\kappa_{1}(\alpha,t)}{\kappa_{0}(\alpha,t)}\neq0$
for all $t\in\mathbb{T}^{\kappa}$.

\begin{theorem}
Suppose equation \eqref{eq14} is regressive and  $a^{2}-4b\neq 0$.
Then, $E_{\lambda_{1}}(\cdot,t_{0})$ and $E_{\lambda_{2}}(\cdot,t_{0})$
form a fundamental system of \eqref{eq14}, where $t_{0}\in \mathbb{T}$
and $\lambda_{1}$ and $\lambda_{2}$ are given as in \eqref{eq16},
and the solution of the initial value problem
\begin{equation}
\label{eq17}
\Delta^{\alpha}\Delta^{\alpha}y(t)+a\Delta^{\alpha}y(t)+by(t)=0,
\quad y(t_{0})=y_{0},
\quad \Delta^{\alpha}y_{t_{0}}=y_{0}^{\alpha},
\end{equation}
is given by
$$
y_{0}(t)=\frac{E_{\lambda_{1}}(t,t_{0})
+E_{\lambda_{2}}(t,t_{0})}{2}
+\frac{ay_{0}+2y_{0}^{\alpha}}{\sqrt{a^{2}-4b}}
\frac{E_{\lambda_{2}}(t,t_{0})-E_{\lambda_{1}}(t,t_{0})}{2},
\quad t\in\mathbb{T}^{\kappa^{2}}_{[t_{0},\infty)}.
$$
\end{theorem}

\begin{proof}
Since $\lambda_{1}$ and $\lambda_{2}$, given by \eqref{eq16},
are solutions of the characteristic equation \eqref{eq15},
we know that both $E_{\lambda_{1}}(\cdot,t_{0})$
and $E_{\lambda_{2}}(\cdot,t_{0})$ are solutions of \eqref{eq14}.
Moreover,
\begin{equation*}
\begin{split}
W\left( E_{\lambda_{1}}(t,t_{0}), E_{\lambda_{2}}(t,t_{0})\right)
&= \det
\left(
\begin{array}{ccc}
E_{\lambda_{1}}(t,t_{0})& E_{\lambda_{2}}(t,t_{0})\\
\lambda_{1}E_{\lambda_{1}}(t,t_{0}) & \lambda_{2}E_{\lambda_{2}}(t,t_{0})\\
\end{array}
\right)\\
&=\lambda_{2}E_{\lambda_{1}}(t,t_{0}) E_{\lambda_{2}}(t,t_{0}) 
-\lambda_{1}E_{\lambda_{1}}(t,t_{0}) E_{\lambda_{2}}(t,t_{0})\\
&=(\lambda_{2}-\lambda_{1} )E_{\lambda_{1}}(t,t_{0}) E_{\lambda_{2}}(t,t_{0})\\
&=\sqrt{a^{2}-4b}E_{\lambda_{1}}(t,t_{0}) E_{\lambda_{2}}(t,t_{0}),
\end{split}
\end{equation*}
so that $W( E_{\lambda_{1}}(t,t_{0}) E_{\lambda_{2}}(t,t_{0}))\neq 0$
for all $t\in\mathbb{T}^{\kappa}_{[t_{0},\infty)}$, unless $a^{2}-4b=0$.
Having obtained a fundamental system $y_{1}=E_{\lambda_{1}}(\cdot,t_{0})$
and $y_{2}= E_{\lambda_{2}}(\cdot,t_{0})$ of \eqref{eq14},
now we obtain a solution of \eqref{eq17}, namely $y(t)=c_{1}y_{1}(t)+c_{2}y_{2}(t)$.
For that we solve the linear system of equations
\begin{equation*}
\begin{cases}
y_{0}=c_{1}y_{1}(t_{0})+c_{2}y_{2}(t_{0})\\
\Delta^{\alpha}y(t_{0})=\lambda_{1}c_{1}y_{1}(t_{0})+\lambda_{2}c_{2}y_{2}(t_{0})
\end{cases}
\end{equation*}
in the unknowns $c_{1}$ and $c_{2}$, obtaining
$c_{1}=\frac{y_{0}}{2}-\frac{a y_{0}+2y_{0}^{\alpha}}{2\sqrt{a^{2}-4b}}$
and $c_{2}=\frac{y_{0}}{2}+\frac{a y_{0}+2y_{0}^{\alpha}}{2\sqrt{a^{2}-4b}}$.
\end{proof}

Hyperbolic functions are associated with the case $a=0$ and $b<0$.

\begin{definition}[Hyperbolic functions]
\label{def:hf}
Let $\mathbb{T}$ be a time scale and $t_{0}\in\mathbb{T}$.
If $p\in C_{rd}$ and $\kappa^{2}_{0}
-2\mu\kappa_{0}\kappa_{1}+\mu^{2}(-p^{2}+\kappa^{2}_{1})\neq 0$
for all $t\in \mathbb{T}^{\kappa}$, then we define
the hyperbolic functions $\cosh_{p\Delta^{\alpha}}(\cdot,t_{0})$
and $\sinh_{p\Delta^{\alpha}}(\cdot,t_{0})$ on $\mathbb{T}_{[t_{0},\infty)}$ by
$$
\cosh_{p\Delta^{\alpha}}(\cdot,t_{0})=\frac{E_{p}(\cdot,t_{0})+E_{-p}(\cdot,t_{0})}{2}
\ \text{ and } \
\sinh_{p\Delta^{\alpha}}(\cdot,t_{0})=\frac{E_{p}(\cdot,t_{0})-E_{-p}(\cdot,t_{0})}{2}.
$$
\end{definition}

\begin{remark}
The condition
$\kappa^{2}_{0}-2\mu\kappa_{0}\kappa_{1}+\mu^{2}(-p^{2}+\kappa^{2}_{1})\neq 0$
of Definition~\ref{def:hf} is equivalent to
$1+\mu(t) \frac{p(t)-\kappa_{1}(\alpha,t)}{\kappa_{0}(\alpha,t)}\neq 0$
and $1-\mu(t)\frac{p(t)+\kappa_{1}(\alpha,t)}{\kappa_{0}(\alpha,t)} \neq 0$.
\end{remark}

\begin{lemma}
\label{lemma:cosh:sinh}
Let $\kappa^{2}_{0}-2\mu\kappa_{0}\kappa_{1}+\mu^{2}(-p^{2}+\kappa^{2}_{1})\neq0$
for all $t\in\mathbb{T}^{\kappa}$. Then,
\begin{gather*}
\Delta^{\alpha}\cosh_{p\Delta^{\alpha}}(\cdot,t_{0})=p\sinh_{p\Delta^{\alpha}}(\cdot,t_{0}),\\ \Delta^{\alpha}\sinh_{p\Delta^{\alpha}}(\cdot,t_{0})=p\cosh_{p\Delta^{\alpha}}(\cdot,t_{0}),\\
\cosh^{2}_{p\Delta^{\alpha}}(\cdot,t_{0})
-\sinh^{2}_{p\Delta^{\alpha}}(\cdot,t_{0})
=E_{p}(\cdot,t_{0})E_{-p}(\cdot,t_{0}),
\end{gather*}
for all $t\in \mathbb{T}_{[t_{0},\infty)}$.
\end{lemma}

\begin{proof}
The first two formulas are trivially verified. The last relation follows from
\begin{equation*}
\begin{split}
(\cosh^{2}_{p\Delta^{\alpha}}&-\sinh^{2}_{p\Delta^{\alpha}})(\cdot,t_{0})\\
&= \left(\frac{E_{p}(\cdot,t_{0})+E_{-p}(\cdot,t_{0})}{2}\right)^{2}
-\left(\frac{E_{p}(\cdot,t_{0})-E_{-p}(\cdot,t_{0})}{2}\right)^{2}\\
&= \frac{E^{2}_{p}(\cdot,t_{0})+2E_{p}(\cdot,t_{0})E_{-p}(\cdot,t_{0})
+E^{2}_{-p}(\cdot,t_{0})}{4}\\
&\qquad -\frac{E^{2}_{p}(\cdot,t_{0})
-2E_{p}(\cdot,t_{0}) E_{-p}(\cdot,t_{0})
+E^{2}_{-p}(\cdot,t_{0})}{4}\\
&= E_{p}(\cdot,t_{0})E_{-p}(\cdot,t_{0})
\end{split}
\end{equation*}
for all $t\in \mathbb{T}_{[t_{0},\infty)}$.
\end{proof}

\begin{theorem}
If $\gamma\in\mathbb{R}\backslash \{0\}$,
$\kappa^{2}_{0}-2\mu\kappa_{0}\kappa_{1}+\mu^{2}(-\gamma^{2}+\kappa^{2}_{1})\neq 0$,
and $t_{0}\in\mathbb{T}^{\kappa}$, then
$$
y(t)=c_{1}\cosh_{\gamma\Delta^{\alpha}}(t,t_{0})+c_{2}\sinh_{\gamma\Delta^{\alpha}}(t,t_{0})
$$
is a general solution of
\begin{equation}
\label{eq18}
\Delta^{\alpha}\Delta^{\alpha}y-\gamma^{2}y=0
\end{equation}
on $t\in\mathbb{T}^{\kappa^{2}}_{[t_{0},\infty)}$,
where $c_{1}$ and $c_{2}$ are arbitrary constants.
\end{theorem}

\begin{proof}
One can easily prove that $\cosh_{\gamma\Delta^{\alpha}}(\cdot,t_{0})$
and $\sinh_{\gamma\Delta^{\alpha}}(\cdot,t_{0})$
are solutions of \eqref{eq18}. We prove that
they form a fundamental set of solutions for \eqref{eq18}:
\begin{equation*}
\begin{split}
W\left(\cosh_{\gamma\Delta^{\alpha}}(t,t_{0}),
\sinh_{\gamma\Delta^{\alpha}}(t,t_{0})\right)
&=\det
\left(
\begin{array}{ccc}
\cosh_{\gamma\Delta^{\alpha}}(t,t_{0})
& \sinh_{\gamma\Delta^{\alpha}}(t,t_{0}))\\
\gamma\sinh_{\gamma\Delta^{\alpha}}(t,t_{0}))
& \gamma\cosh_{\gamma\Delta^{\alpha}}(t,t_{0})\\
\end{array}
\right)\\
&= \gamma\cosh^{2}_{\gamma\Delta^{\alpha}}(t,t_{0})
-\gamma\sinh^{2}_{\gamma\Delta^{\alpha}}(t,t_{0})\\
&=\gamma \left(\cosh^{2}_{\gamma\Delta^{\alpha}}(t,t_{0})
-\sinh^{2}_{\gamma\Delta^{\alpha}}(t,t_{0})\right)\\
&=\gamma E_{\gamma}(t,t_{0}) E_{-\gamma}(t,t_{0})
\end{split}
\end{equation*}
is different from zero for all $t\in\mathbb{T}^{\kappa}_{[t_{0},\infty)}$,
unless $\gamma=0$.
\end{proof}

\begin{example}
Let $\mathbb{T}$ be a time scale, $t_{0}\in\mathbb{T}^{\kappa}$.
If $\kappa^{2}_{0}-2\mu\kappa_{0}\kappa_{1}+\mu^{2}(-\gamma^{2}+\kappa^{2}_{1})\neq0$,
with $\gamma\in\mathbb{R}\backslash \{0\}$, then the solution of the initial value problem
$$
\Delta^{\alpha}\Delta^{\alpha}y(t)-\gamma^{2}y(t)=0,
\quad y(t_{0})=y_{0},
\quad \Delta^{\alpha}y(t_{0})=y_{0}^{\alpha},
$$
is given by
$$
y(t)=y_{0}\cosh_{\gamma\Delta^{\alpha}}(t,t_{0})
+\frac{y_{0}^{\alpha}}{\gamma}\sinh_{\gamma\Delta^{\alpha}}(t,t_{0})
$$
for all $t\in\mathbb{T}^{\kappa^{2}}_{[t_{0},\infty)}$.
\end{example}

\begin{definition}[Trigonometric functions]
\label{def:tf}
Let $\mathbb{T}$ be a time scale, $p\in C_{rd}$, $t_{0}\in\mathbb{T}$,
and $\kappa^{2}_{0}-2\mu\kappa_{0}\kappa_{1}+\mu^{2}(p^{2}+\kappa^{2}_{1})\neq0$
for all $t\in \mathbb{T}^{\kappa}$. Then we define the trigonometric functions
$\cos_{p\Delta^{\alpha}}(\cdot,t_{0})$ and $\sin_{p\Delta^{\alpha}}(\cdot,t_{0})$ by
$$
\cos_{p\Delta^{\alpha}}(\cdot,t_{0})=\frac{E_{ip}(\cdot,t_{0})+E_{-ip}(\cdot,t_{0})}{2}
\ \text{ and } \
\sin_{p\Delta^{\alpha}}(\cdot,t_{0})=\frac{E_{ip}(\cdot,t_{0})-E_{-ip}(\cdot,t_{0})}{2i}.
$$
\end{definition}

\begin{remark}
The condition
$\kappa^{2}_{0}-2\mu\kappa_{0}\kappa_{1}+\mu^{2}(p^{2}+\kappa^{2}_{1})\neq0$
of Definition~\ref{def:tf} is equivalent to
$1+\mu(t)\left(\frac{ip(t)-\kappa_{1}(\alpha,t)}{\kappa_{0}(\alpha,t)}\right)\neq 0$
and $1-\mu(t)\frac{i p(t)+\kappa_{1}(\alpha,t)}{\kappa_{0}(\alpha,t)}\neq 0$.
\end{remark}

\begin{lemma}
Let $p\in C_{rd}$. If
$\kappa^{2}_{0}-2\mu\kappa_{0}\kappa_{1}+\mu^{2}(p^{2}+\kappa^{2}_{1})\neq 0$
for all $t\in \mathbb{T}^{\kappa}$, then
\begin{gather*}
\Delta^{\alpha}\cos_{p\Delta^{\alpha}}(\cdot,t_{0})=-p\sin_{p\Delta^{\alpha}}(\cdot,t_{0}),\\
\Delta^{\alpha}\sin_{p\Delta^{\alpha}}(\cdot,t_{0})=p\cos_{p\Delta^{\alpha}}(\cdot,t_{0}),\\
\cos^{2}_{p\Delta^{\alpha}}(\cdot,t_{0})+\sin^{2}_{p\Delta^{\alpha}}(\cdot,t_{0})
=E_{ip}(\cdot,t_{0})E_{-ip}(\cdot,t_{0}).
\end{gather*}
\end{lemma}

\begin{remark}
If $\alpha=1$, then
$E_{ip}(\cdot,t_{0})=e_{ip}(\cdot,t_{0})=\cos_{p}(\cdot,t_{0})+i\sin_{p}(\cdot,t_{0})$.
\end{remark}

\begin{proof}
Similarly to Lemma~\ref{lemma:cosh:sinh}, the first two formulas
are easily verified. We have
\begin{equation*}
\begin{split}
\cos^{2}_{p\Delta^{\alpha}}&(\cdot,t_{0})+\sin^{2}_{p\Delta^{\alpha}}(\cdot,t_{0})\\
&= \left(\frac{E_{ip}(\cdot,t_{0})+E_{-ip}(\cdot,t_{0})}{2}\right)^{2}
+\left(\frac{E_{ip}(\cdot,t_{0})-E_{-ip}(\cdot,t_{0})}{2i}\right)^{2}\\
&=\frac{E^{2}_{ip}(\cdot,t_{0})+2E_{ip}(\cdot,t_{0})E_{-ip}(\cdot,t_{0})
+E^{2}_{-ip}(\cdot,t_{0})}{4}\\
&\qquad -\frac{E^{2}_{ip}(\cdot,t_{0})-2E_{ip}(\cdot,t_{0})E_{-ip}(\cdot,t_{0})
+E^{2}_{-ip}(\cdot,t_{0})}{4}\\
&=E_{ip}(\cdot,t_{0})E_{-ip}(\cdot,t_{0})
\end{split}
\end{equation*}
and the last relation also holds.
\end{proof}

\begin{example}
Let $\mathbb{T}=\mathbb{R}$, $\gamma\in\mathbb{R}$, and $t_{0}\in\mathbb{T}$.
Then, the conformable trigonometric functions cosine and sine are given by
\begin{equation*}
\begin{split}
\cos_{\gamma\Delta^{\alpha}}(t,t_{0})
&=\frac{E_{i\gamma}(t,t_{0})+E_{-i\gamma}(t,t_{0})}{2}\\
&=\frac{e^{\int_{t_{0}}^{t}\frac{i\gamma-\kappa_{1}(s,t_{0})}{\kappa_{0}(s,t_{0})}ds}
+e^{\int_{t_{0}}^{t}\frac{-i\gamma-\kappa_{1}(s,t_{0})}{\kappa_{0}(s,t_{0})}ds}}{2}\\
&=\frac{e^{i\int_{t_{0}}^{t}\frac{\gamma}{\kappa_{0}(s,t_{0})}ds}
e^{-\int_{t_{0}}^{t}\frac{\kappa_{1}(s,t_{0})}{\kappa_{0}(s,t_{0})}ds}
+e^{-i\int_{t_{0}}^{t}\frac{\gamma}{\kappa_{0}(s,t_{0})}ds}
e^{-\int_{t_{0}}^{t}\frac{\kappa_{1}(s,t_{0})}{\kappa_{0}(s,t_{0})}ds}}{2}\\
&=\frac{e^{-\int_{t_{0}}^{t}\frac{\kappa_{1}(s,t_{0})}{\kappa_{0}(s,t_{0})}ds}
\left(2\cos(\int_{t_{0}}^{t}\frac{\gamma}{\kappa_{0}(s,t_{0})}ds )\right)}{2}\\
&=e^{-\int_{t_{0}}^{t}\frac{\kappa_{1}(s,t_{0})}{\kappa_{0}(s,t_{0})}ds}
\cos\left(\int_{t_{0}}^{t}\frac{\gamma}{\kappa_{0}(s,t_{0})}ds\right)
\end{split}
\end{equation*}
and
\begin{equation*}
\begin{split}
\sin_{\gamma\Delta^{\alpha}}(t,t_{0})
&=\frac{E_{i\gamma}(t,t_{0})-E_{-i\gamma}(t,t_{0})}{2i}\\
&=\frac{e^{\int_{t_{0}}^{t}\frac{i\gamma-\kappa_{1}(s,t_{0})}{\kappa_{0}(s,t_{0})}ds}
-e^{\int_{t_{0}}^{t}\frac{-i\gamma-\kappa_{1}(s,t_{0})}{\kappa_{0}(s,t_{0})}ds}}{2i}\\
&=\frac{e^{i\int_{t_{0}}^{t}\frac{\gamma}{\kappa_{0}(s,t_{0})}ds}
e^{-\int_{t_{0}}^{t}\frac{\kappa_{1}(s,t_{0})}{\kappa_{0}(s,t_{0})}ds}
-e^{-i\int_{t_{0}}^{t}\frac{\gamma}{\kappa_{0}(s,t_{0})}ds}
e^{-\int_{t_{0}}^{t}\frac{\kappa_{1}(s,t_{0})}{\kappa_{0}(s,t_{0})}ds}}{2i}\\
&=e^{-\int_{t_{0}}^{t}\frac{\kappa_{1}(s,t_{0})}{\kappa_{0}(s,t_{0})}ds}
\sin\left(\int_{t_{0}}^{t}\frac{\gamma}{\kappa_{0}(s,t_{0})}ds\right).
\end{split}
\end{equation*}
\end{example}

\begin{theorem}
Let $\mathbb{T}$ be a time scale and $t_{0}\in\mathbb{T}^{\kappa}$.
If $\kappa^{2}_{0}-2\mu\kappa_{0}\kappa_{1}+\mu^{2}(\gamma^{2}+\kappa^{2}_{1})\neq0$,
$\gamma\in\mathbb{R}\backslash \{0\}$, then
$y(t)=c_{1}\cos_{\gamma\Delta^{\alpha}}(t,t_{0})+c_{2}\sin_{\gamma\Delta^{\alpha}}(t,t_{0})$
is a general solution of
\begin{equation}
\label{eq19}
\Delta^{\alpha}\Delta^{\alpha}y+\gamma^{2}y=0,
\quad t\in\mathbb{T}^{\kappa^{2}}.
\end{equation}
\end{theorem}

\begin{proof}
One can easily show that $\cos_{\gamma\Delta^{\alpha}}(\cdot, t_0)$
and $\sin_{\gamma\Delta^{\alpha}}(\cdot, t_0)$ are solutions of \eqref{eq19}.
We prove that they form a fundamental set of solutions for \eqref{eq19}:
for $\gamma \neq 0$,
\begin{equation*}
\begin{split}
W&\left(\cos_{\gamma\Delta^{\alpha}}(t, t_{0}),
\sin_{\gamma \Delta^{\alpha}}(t, t_{0})\right)
= \det
\left(
\begin{array}{ccc}
\cos_{\gamma\Delta^{\alpha}}(t,t_{0})
& \sin_{\gamma\Delta^{\alpha}}(t,t_{0}))\\
-\gamma\sin_{\gamma\Delta^{\alpha}}(t,t_{0}))
& \gamma\cos_{\gamma\Delta^{\alpha}}(t,t_{0})\\
\end{array}
\right)\\
&=\gamma\left(\cos^{2}_{\gamma\Delta^{\alpha}}(t,t_{0})
+\sin^{2}_{\gamma\Delta^{\alpha}}(t,t_{0})\right)^{2}
=\gamma E_{ip}(t,t_{0})E_{-ip}(t,t_{0}) \neq 0
\end{split}
\end{equation*}
for all $t\in \mathbb{T}^{\kappa}_{[t_{0},\infty)}$.
We conclude that $y(t)=c_{1}\cos_{\gamma\Delta^{\alpha}}(t,t_{0})
+c_{2}\sin_{\gamma\Delta^{\alpha}}(t,t_{0})$,
$t \in \mathbb{T}_{[t_{0},\infty)}$, is a general solution of \eqref{eq19}.
\end{proof}


\par\bigskip\noindent
{\bf Acknowledgment.} This research is part of first author's Ph.D.,
which is carried out at Sidi Bel Abbes University, Algeria.
It was essentially finished while Bayour and Hamoudi were visiting the Department
of Mathematics of University of Aveiro, Portugal, 2016.
The hospitality of the host institution and the financial support
of Universities of Chlef and Ain-Temouchent, Algeria, 
are here gratefully acknowledged. Torres was supported 
by Portuguese funds through CIDMA and FCT, within project 
UID/MAT/04106/2013. The authors are very grateful to 
Salima Hassani and to an anonymous Referee, for several 
questions, remarks and suggestions.



\end{document}